\documentclass[a4paper,11pt]{article}

\usepackage[english]{babel}
\usepackage[utf8]{inputenc}
\usepackage{amsmath}
\usepackage{amssymb}
\usepackage{graphicx}
\usepackage[colorinlistoftodos]{todonotes}
\usepackage[a4paper]{geometry}

\usepackage{enumerate}
\newtheorem{theorem}{Theorem}[section]
\newtheorem{proposition}[theorem]{Proposition}
\newtheorem{lemma}[theorem]{Lemma}

\newtheorem{remark}[theorem]{Remark}
\newtheorem{example}[theorem]{Example}

\newcommand{\N}{\mathbb{N}}
\newcommand{\R}{\mathbb{R}}
\newenvironment{proof}{{\bf Proof:}}{\hfill $\square$}

\begin{document}

\newgeometry{left = 3cm, right = 3cm, top = 3cm, bottom = 3cm}

\title{Barabanov norms, Lipschitz continuity and monotonicity for the max algebraic joint spectral radius}

\author{Nicola Guglielmi \thanks{Dipartimento di Ingegneria Scienze Informatiche e Matematica,
	Universita' degli Studi di L' Aquila, Italy}, Oliver Mason\thanks{Corresponding Author. Dept. of Mathematics and Statistics/Hamilton Institute, Maynooth University-National University of Ireland Maynooth, Maynooth, Co. Kildare, Ireland. email: oliver.mason@nuim.ie} and Fabian Wirth \thanks{University of Passau, 
		Faculty of Computer Science and Mathematics, Chair of Dynamical Systems, Germany}}

\date{\today}
\maketitle 

\begin{abstract}
We present several results describing the interplay between the max algebraic joint spectral radius (JSR) for compact sets of matrices and suitably defined matrix norms.  In particular, we extend a classical result for the conventional algebra, showing that the JSR can be described in terms of induced norms of the matrices in the set.  We also show that for a set generating an irreducible semigroup (in a cone-theoretic sense), a monotone Barabanov norm always exists.  This fact is then used to show that the max algebraic JSR is locally Lipschitz continuous on the space of compact irreducible sets of matrices with respect to the Hausdorff distance.  We then prove that the JSR is Hoelder continuous on the space of compact sets of nonnegative matrices.  Finally, we prove a strict monotonicity property for the max algebraic JSR that echoes a fact for the classical JSR.
\end{abstract}

\begin{center}
\end{center}

\section{Introduction}
The joint spectral radius (JSR) plays a key role in a variety of areas, including the stability theory of difference and differential inclusions and wavelet analysis.  There is now a substantial body of work on the JSR and certain key questions and themes in its study are established.  For our purposes, the line of work concerning the existence and properties of special types of norms associated with the JSR \cite{Wirth1, Bar} is of particular relevance.  Specifically, we will be concerned with extremal and Barabanov norms; the latter are known to exist for compact sets of matrices that are irreducible in the representation-theoretic sense, meaning that the matrices have no non-trivial common invariant subspace.  Barabanov and extremal norms play a key role in the proofs of regularity and continuity for the JSR given in \cite{Wirth1, Wirth2}.  Recent work has shown that when the inclusion is positive with respect to a proper cone, it is possible to prove the existence of an extremal norm under the less restrictive condition of cone-irreducibility \cite{MasWir1}.  In particular, this is true for compact sets of nonnegative matrices and this fact was applied to differential inclusion models in epidemiology in \cite{Boketal14}.  

In this paper, we extend the work described in the previous paragraph to the setting of the max-algebra.  In keeping with \cite{Bapat} and other works in the field, we consider $\mathbb{R}_+$ equipped with the two operations: $a \oplus b = \max\{a, b\}$; $a \otimes b = ab$.  These operations extend to matrices and vectors with nonnegative entries in direct analogy with conventional linear algebra.  The joint spectral radius can be defined over the max algebra in a manner analogous to conventional algebra \cite{Pep1, Pep2,Lur2} and certain key results still hold in the new setting.  In particular, it has been shown that the Berger-Wang formula holds over the max algebra and that the JSR is continuous with respect to Hausdorff distance.  While max-algebraic induced norms have been introduced and studied in \cite{Lur1}, to date there has been no substantial work done on characterising extremal or Barabanov norms in this setting.  Furthermore, while continuity has been established for the JSR, it has not yet been shown that the max-algebraic JSR is Lipschitz continuous.  We will address both of the above questions in this paper, showing that Barabanov norms always exist for cone-irreducible sets of matrices and that the JSR is indeed Lipschitz continuous over the max algebra.  We will present a result on the monotonicity of the max algebraic JSR in the spirit of the work of \cite{Wirth2} . 

The outline and structure of the paper is as follows.  In the next section, we introduce our principal notation as well as recalling background on relevant concepts and results concerning matrices over the max algebra.  In Section \ref{sec:MaxJSR}, we recall the definition of the joint spectral radius for sets of matrices over the max algebra and introduce the concept of an irreducible semigroup of matrices in this setting.  We move on in Section \ref{sec:norms} to introduce extremal and Barabanov norms for max algebraic semigroups.  We show here that an irreducible semigroup always admits a Barabanov norm and, moreover, we explicitly describe such a norm for the case where it exists.  We then use this class of Barabanov norms to prove that the joint spectral radius is locally Lipschitz continuous on the space of irreducible, compact sets of nonnegative matrices.  This result is then strengthened further in Section \ref{sec:hoelder} where we show that the JSR is in fact Hoelder continuous on the space of compact subsets of $\mathbb{R}^{n \times n}_+$ (removing the conditions of irreducibility or $\mu(\Psi) > 0$ required for Lipschitz continuity), using the fact that the max algebraic joint spectral radius for a set of matrices is given by the max algebraic spectral radius of a single matrix associated with the set.  In Section \ref{sec:mon} we prove a monotonicity property for the max algebraic JSR before finally giving our concluding remarks in Section \ref{sec:conc}. 

\section{Background}
\label{sec:bg}
Throughout the paper, $\mathbb{R}^n$ and $\mathbb{R}^{n \times n}$ denotes the vector spaces of $n$-tuples of real numbers and of $n \times n$ real matrices respectively.  $\mathbb{R}^n_+$ ($\mathbb{R}^{n \times n}_+$) denote the cones of vectors (matrices) with nonnegative entries.  As usual, for $A \in \mathbb{R}^{n \times n}$, $A^T$ denotes the transpose of $A$.  We adopt the convention that $x \in \mathbb{R}^n$ is represented as a column vector with $x^T$ denoting the equivalent row-vector.  

We denote by $\mathbf{1}_{n \times n}$ the $n \times n$ matrix with all entries equal to one.  For a vector $x \in \mathbb{R}^n$, we denote by $|x|$ the vector whose $i$th component is given by $|x_i|$.  

We shall later work with absolute and monotone norms on $\mathbb{R}^{n}$.  Recall that a norm $\|\cdot\|$ is absolute if $\|x\| = \| |x|\|$ for all $x \in \mathbb{R}^n$.  The norm $\|\cdot\|$ is monotone if $|x| \leq |y|$ implies $\| x \| \leq \|y \|$ for $x, y$ in $\mathbb{R}^n$.  It is well known that a norm on $\mathbb{R}^n$ is absolute if and only if it is monotone. 

Given two vectors $x, y$ in $\mathbb{R}^n$, the notation $x \geq y$ indicates that $x_i \geq y_i$ for $1 \leq i \leq n$; $x > y$ means that $x \geq y$, $x \neq y$; $x \gg y$ means that $x_i > y_i$ for $1 \leq i\leq n$.  Analogous notation is also used for matrices. 

Given $A \in \mathbb{R}^{n \times n}_+$, the weighted directed graph $D(A)$ consists of the nodes $\{1, \ldots, n\}$ with an edge from $i$ to $j$ if and only if $a_{ij} > 0$.  The \emph{weight} of the edge $(i, j)$ is then given by $a_{ij}$.

We say that a matrix $A \in \mathbb{R}^{n \times n}$ is \emph{irreducible} if there is no non-trivial subset $I \subset \{1, \ldots n\}$ such that $a_{ij} = 0$ for all $i \in I$, $j \not\in I$.  $A$ is irreducible if and only if $D(A)$ is strongly connected.

\emph{Max algebraic spectral theory}

The \emph{max algebra} consists of the set $\mathbb{R}_+$ of nonnegative real numbers equipped with the two operations 
$$a\oplus b = \max\{a, b\}; \; a \otimes b = ab.$$
These operations can be extended to vectors and matrices with nonnegative entries in the obvious fashion and we use the same notation to denote the matrix-vector operations.  We denote by $A_\otimes^p$ ($p \in \mathbb{N}$) the $p$th max algebraic power of a matrix $A \in \mathbb{R}^{n \times n}_+$. 

The \emph{maximal cycle geometric mean} $\mu(A)$ of $A \in \mathbb{R}^{n \times n}_+$ plays a central role in the spectral theory of the max algebra.  This can be defined as 
\begin{equation}
\label{eq:mu1} \mu(A) = \max \{(a_{i_1i_2}a_{i_2i_3}\cdots a_{i_ki_1})^{1/k} \mid k \geq 1, 1\leq i_1, i_2, \ldots, i_k \leq n \}
\end{equation}
where $i_1, \ldots, i_k$ are distinct indices.  This describes the maximal geometric mean weight of the cycles in the graph $D(A)$.  It is not hard to see that $\mu(A) = \mu(A^T)$.

There is a well-developed spectral theory for matrices over the max algebra \cite{Els1, Bapat}, echoing the results of classical Perron-Frobenius theory for nonnegative matrices.  Given $A \in \mathbb{R}^{n \times n}_+$, $\lambda \geq 0$ is a max-eigenvalue of $A$ if there is some non-zero $x \in \mathbb{R}^n_+$ such that $$A \otimes x = \lambda x.$$  The vector $x$ is then a max eigenvector of $A$. 

For our purposes, the facts recorded in the following proposition will prove sufficient.

\begin{proposition}
\label{prop:PF1} Let $A \in \mathbb{R}^{n \times n}_+$ be given.  Then:
\begin{itemize}
\item[(i)] $\mu(A)$ is the largest max eigenvalue of $A$;
\item[(ii)] if $A$ is irreducible, $\mu(A)$ is the unique max eigenvalue of $A$ and all max eigenvectors $x$ of $A$ satisfy $x \gg 0$.  
\end{itemize}
\end{proposition}

\emph{Convex hulls and max-linear spans}

Motivated by problems in timed discrete event systems and combinatorial optimisation, several authors have studied notions of convexity in the max algebraic (or the isomorphic max-plus and min-plus) settings.  We will consider the following definition of convex hull which is that given in  \cite{Gaub2}.  For more background and examples on max algebraic and so-called \emph{tropical} convexity, see \cite{Gaub2, DevStur}.  

Given a subset $M$ of either $\mathbb{R}^n_+$ or $\mathbb{R}^{n \times n}_+$, we define the max-convex hull of $M$ to be
\begin{equation}
\label{eq:maxconvdef} \textrm{conv}_{\otimes} (M) := \left\{ \bigoplus_{i=1}^k \alpha_i x_i \mid k\in\mathbb{N}, x_i \in M, \alpha_i \geq 0, i=1,\ldots,k, \bigoplus_{i=1}^k \alpha_i  = 1\right\}.
\end{equation}

For example in Figure \ref{fig:trop1}, the lines in bold together with the shaded region denote the max convex hull of the points $a, b, c$.  
\begin{figure}
\centering
\includegraphics[width=8cm]{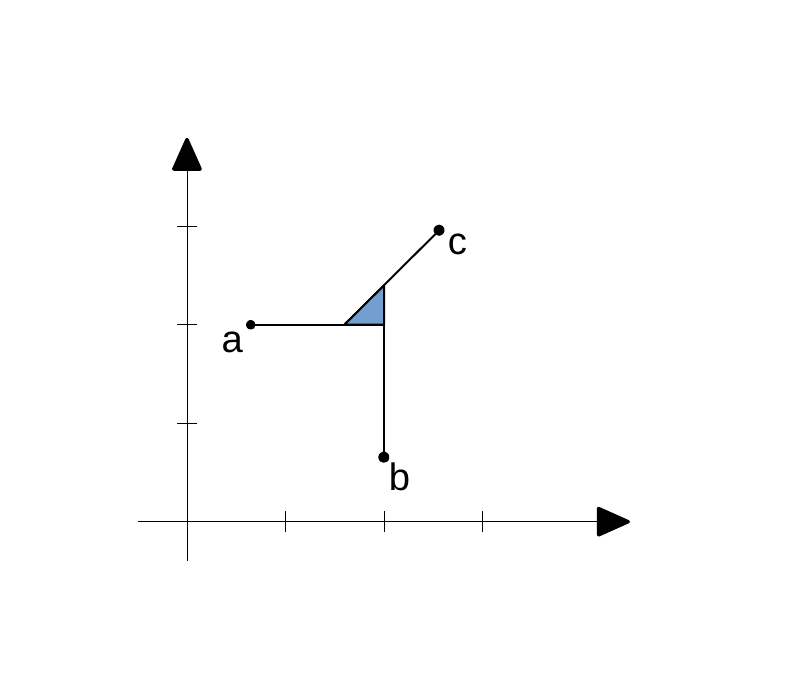}
\caption{A simple max algebraic convex hull}
\label{fig:trop1}
\end{figure}
The max algebraic analogue of a vector space spanned by $M$ is given by a max cone \cite{SergeiCone}.  Formally:
\begin{equation}
\label{eq:maxspandef} \textrm{span}_{\otimes} (M) := \left\{ \bigoplus_{i=1}^k \alpha_i x_i \mid k\in\mathbb{N}, x_i \in M, \alpha_i \geq 0, i=1,\ldots,k \right\}.
\end{equation}

Throughout the paper, we will assume that $\mathbb{R}^n_+$ is equipped with the relative topology inherited from the norm topology on $\mathbb{R}^n$.  

\section{The max-algebraic joint spectral radius - elementary properties}
\label{sec:MaxJSR}

Our particular interest is in extending results on the joint spectral radius, in the spirit of the papers \cite{Pep1, Pep2, Lur1, Lur2}.

\emph{The max-algebraic joint spectral radius}

Given a bounded set $\Psi \subseteq \mathbb{R}^{n \times n}_+$ of nonnegative matrices and an integer $m \geq 1$, we define the set
$$\Psi_{\otimes}^m := \{A_{i_1}\otimes A_{i_2} \otimes \cdots \otimes A_{i_m}\mid A_{i_j} \in \Psi, 1\leq j \leq m\}$$
consisting of all products of length $m$ formed from the matrices of $\Psi$. 

The semigroup $\mathcal{S}(\Psi)$ associated with $\Psi$ is then defined by 
$$\mathcal{S}(\Psi) := \bigcup_{m\geq 0} \Psi_{\otimes}^m,$$
where we set $\Psi_{\otimes}^0 = \{I\}$. 

The max algebraic generalised spectral radius $\mu(\Psi)$ was introduced in \cite{Lur1} and subsequently studied in a number of papers \cite{Lur2, Pep1, Pep2, Pep3}; it is defined by the formula:
\begin{equation}
\label{GSR} \mu(\Psi) : = \limsup_{m \rightarrow \infty} \left(\sup_{A \in \Psi_{\otimes}^m} \mu(A)\right)^{1/m}.
\end{equation}
Using a variety of techniques, it has been shown in \cite{Pep1, Lur1, Pep3} that the celebrated Berger-Wang formula also holds in the max algebra; formally:
\begin{equation}
\label{eq:BW} \mu(\Psi) = \lim_{m \rightarrow \infty} \left(\sup_{A \in \Psi_{\otimes}^m} \|A\|\right)^{1/m},
\end{equation}
where $\|\cdot \|$ can be any norm on $\mathbb{R}^{n \times n}$.  The quantity given by the right hand side of \eqref{eq:BW} is a max algebraic analogue of the joint spectral radius (JSR).  As the two quantities are equal, we will simply use the terminology joint spectral radius (JSR) in this paper.

One of the most striking contrasts between the max-algebraic and classical settings is that the max algebraic JSR is given by the max-algebraic spectral radius of a single matrix.  If we define 
\begin{equation}\label{eq:S}
S(\Psi) = \bigoplus_{A \in \Psi} A
\end{equation}
so that $S(\Psi)=[s_{ij}]_{i,j=1}^n$ with
$$s_{ij} = \sup_{A \in \Psi} a_{ij},$$
then 
\begin{equation}
\label{eq:Sres}
\mu(\Psi) = \mu(S(\Psi)).
\end{equation}
This result was first shown for finite $\Psi$ (albeit in the isomorphic max-plus setting) in \cite{Gaub} and then linear algebraic proofs for the finite and bounded cases were provided in \cite{Buket} and \cite{Pep3} respectively. 

Note that as $S \in \textrm{cl}\textrm{ conv}_{\otimes}(\Psi)$, it follows from \eqref{eq:Sres} that for any bounded set $\Psi$, 
\begin{equation}
\label{eq:clconv}
\mu(\textrm{cl}\textrm{ conv}_{\otimes}(\Psi)) = \mu(\Psi).
\end{equation}

\subsection{Elementary Properties}

A basic property of the JSR in classical algebra is that it is invariant under common similarity transformations.  Before showing that the analogous result holds for the max algebra, we recall the characterisation of invertible matrices in the max algebra \cite{SergeiInv}.

\begin{lemma}
\label{lem:invert} $P \in \mathbb{R}^{n \times n}_{+}$ is max-invertible if and only if there is a vector $v \gg 0$ and a permutation $\sigma$ of $\{1, \ldots, n\}$ such that
$$p_{ij} = \begin{cases}
		v_i & \mbox{if }\, j = \sigma(i)\\
        0 & \mbox{otherwise}.
 	\end{cases}$$
\end{lemma}

If $P$ is invertible, it is not difficult to see that the inverse $Q = P_{\otimes}^{-1}$ is given by
$$q_{ij} = \begin{cases}
		\frac{1}{v_j} & \mbox{if }\, j = \sigma^{-1}(i)\\
        0 & \mbox{otherwise}.
 	\end{cases}$$
It is now relatively straightforward to establish the following fact.

\begin{proposition}
\label{prop:simil} Let $\Psi \subseteq \mathbb{R}^{n \times n}_+$ be compact and let $P \in \mathbb{R}^{n \times n}_+$ be max-invertible.  Then 
$$\mu(P\otimes \Psi \otimes P^{-1}) = \mu(\Psi).$$
\end{proposition}
\begin{proof} To see this, we note that for any $A \in \Psi$, the $i, j$ entry of $P\otimes A \otimes P^{-1}$ is given by 
$$\frac{v_i}{v_j}a_{\sigma(i), \sigma(j)}.$$
It follows that if we write 
$$\hat{S} = \bigoplus_{A \in P\otimes \Psi \otimes P^{-1}} A$$
then
\begin{equation}
\label{eq:hats} \hat{s}_{ij} = \frac{v_i}{v_j}s_{\sigma(i), \sigma(j)} 
\end{equation}
where $S$ is given by \eqref{eq:S}.  It is now straightforward to see that the permutation $\sigma$ defines a bijective correspondence between the cycles in $D(S)$ and those in $D(\hat{S})$ and that, moreover, the weight $\pi(C)$ of a cycle in $D(S)$ will be the same as the weight $\pi(\sigma(C))$ of the corresponding cycle in $D(\hat{S})$.  It follows immediately that 
$$\mu(S) = \mu(\hat{S})$$ and the result now follows from \eqref{eq:Sres}. 
\end{proof}

\emph{Irreducible Semigroups}

In keeping with the terminology adopted in \cite{MasWir1}, we say that the semigroup $\mathcal{S}(\Psi)$ is \emph{irreducible} if $\textrm{conv}_{\otimes}(\Psi_{\otimes}^m)$ contains an irreducible matrix for some $m \geq 0$.  

We next show that the semigroup $\mathcal{S}$ is irreducible if and only if the max-convex hull of $\Psi$ contains an irreducible matrix.  This provides a max algebraic version of Proposition 2.3 in \cite{MasWir1}.

\begin{lemma}
\label{lem:irred} Let $\Psi \subseteq \mathbb{R}^{n \times n}_+$ be a compact set of matrices.  The semigroup $\mathcal{S}(\Psi)$ is irreducible if and only if $\textrm{conv}(\Psi)$ contains an irreducible matrix.
\end{lemma}
\textbf{Proof:} It is immediate from the definition that if $\textrm{conv}_{\otimes}(\Psi)$ contains an irreducible element, then the semigroup $\mathcal{S}$ is irreducible.  Conversely, suppose that there is some $m > 0$ such that $\textrm{conv}_{\otimes}(\Psi_{\otimes}^m)$ contains an irreducible matrix $A$.  Then there exist matrices $A_1, \ldots A_p$ in $\Psi_{\otimes}^m$ and real numbers $\alpha_1, \ldots , \alpha_p$ with $\alpha_i \geq 0$, $\bigoplus_i \alpha_i = 1$ such that:
\begin{eqnarray*}
A &=& \alpha_1 A_1 \oplus \cdots \oplus \alpha_p A_p \\
&\leq& \alpha_1 S_{\otimes}^m \oplus \cdots \oplus \alpha_p S_{\otimes}^m \\
&=& S_{\otimes}^m.
\end{eqnarray*}
It follows immediately that $S_{\otimes}^m$ and hence $S$ is irreducible.  As $S \in \textrm{conv}_{\otimes}(\Psi)$ as $\Psi$ is closed, the result follows. $\square$

\section{Extremal Norms, Irreducibility and Lipschitz Continuity}
\label{sec:norms}
In this section, we recall \cite{Lur1} the max-algebraic version of an induced matrix norm and develop our results on extremal and Barabanov norms in the max algebraic setting. 

\emph{Max algebraic induced norms}

Given a norm $\|\cdot\|$ on $\mathbb{R}^n$, we follow \cite{Lur1} and define the max-induced norm $\eta_{\|\cdot \|}(A)$ of $A \in \mathbb{R}^{n \times n}_+$ as 
\begin{equation}
\label{eq:normind} \eta_{\|\cdot \|}(A) := \max_{x >0}\frac{\|A\otimes x\|}{\|x\|}.
\end{equation}
We shall use the notation $\|A\|$ to denote the usual induced matrix norm (constructed using conventional algebra). 

We now recall some useful facts from \cite{Lur1, Lur2}. 

%\begin{lemma} \cite{Lur1}
%\label{lem:LurNorm} Let $A \in \mathbb{R}^{n \times n}_+$ be given and let $\|\cdot \|$ be a norm on $\mathbb{R}^n$. Then 
%$$\mu(A) \leq \eta_{\|\cdot \|}(A).$$
%\end{lemma}

\begin{lemma}\cite{Lur1}
\label{lem:monnorm} Let $\Psi \subseteq \mathbb{R}^{n \times n}_+$ be bounded.  The semigroup $\mathcal{S}(\Psi)$ is bounded if and only if there is a monotone norm $\|\cdot\|$ on $\mathbb{R}^n$ such that $\eta_{\|\cdot \|}(A) \leq 1$ for all $A \in \Psi$. 
\end{lemma}

\begin{lemma}\cite{Lur1}
\label{lem:bounded} Let $\Psi \subseteq \mathbb{R}^{n \times n}_+$ be bounded. If $\mu(A) \leq 1$ for all $A \in \mathcal{S}(\Psi)$, then $\mathcal{S}(\Psi)$ is bounded.  
\end{lemma}

\begin{lemma}\cite{Lur2}
\label{lem:upperlower} Let $\Psi \subseteq \mathbb{R}^{n \times n}_+$ be bounded and let $\|\cdot\|$ be any norm on $\mathbb{R}^n$.  Then for all $m \geq 1$, 
\begin{equation}
\label{eq:upperlower} \left[ \sup_{A \in \Psi_{\otimes}^m} \mu(A) \right]^{1/m} \leq \mu(\Psi) \leq \left[ \sup_{A \in \Psi_{\otimes}^m} \eta_{\|\cdot \|}(A) \right]^{1/m} 
\end{equation}
\end{lemma}

The above results allow us to establish the following fact which provides a max-algebraic version of a classical result concerning the standard joint spectral radius. 

\begin{proposition}
\label{prop:Norm} Let $\Psi$ be a bounded subset of $\mathbb{R}^{n \times n}_+$.  Then
\begin{equation}
\label{eq:normchar} \mu(\Psi) = \inf_{\|\cdot\|} \sup_{A \in \Psi} \eta_{\|\cdot \|}(A).
\end{equation}
The infimum is over all norms $\|\cdot \|$ on $\mathbb{R}^n$.
\end{proposition}
\textbf{Proof:}  It follows immediately from Lemma \ref{lem:upperlower} by taking $m=1$ that 
\begin{equation}\label{eq:nchar1}
\mu(\Psi) \leq \inf_{\|\cdot\|} \sup_{A \in \Psi} \eta_{\|\cdot \|}(A).
\end{equation}
For the converse, we first consider the case $\mu(\Psi) = 0$ and note that this implies that $\mu(S) = 0$ by \eqref{eq:Sres}; hence $\mu(S^T) = 0$ also.  Let $\epsilon > 0$ be given.  Choose $\delta > 0$ so that $\mu(S_{\delta}) < \epsilon$ where $S_{\delta} = S + \delta \mathbf{1}_{n \times n}$.  Clearly $S_{\delta}$ is irreducible and it follows from Proposition~\ref{prop:PF1} that there is a vector $v \gg 0$ such that $$S_{\delta}^T \otimes v = \mu(S_{\delta}^T)v < \epsilon v.$$  

Now consider the monotone norm defined by setting $\|x \|_{\epsilon} = v^T \otimes x$ for $x \in \mathbb{R}^n_+$ and defining $\|x\|_{\epsilon} = \| |x|\|_{\epsilon}$ for $x \not \in \mathbb{R}^n_+$.  A straightforward calculation shows that 
\begin{equation}\label{eq:etaeps1}
\|A \otimes x \|_{\epsilon} \leq \|S \otimes x \|_{\epsilon} \leq \epsilon \|x\|_{\epsilon}
\end{equation}
for all $x \in \mathbb{R}^n_+$, $A \in \Psi$ and hence that 
$$\eta_{\|\cdot\|_{\epsilon}}(A) \leq \epsilon.$$
As $\epsilon > 0$ was arbitrary, we can conclude that 
$$\inf_{\|\cdot\|} \sup_{A \in \Psi} \eta_{\|\cdot \|}(A) = 0$$ so that \eqref{eq:normchar} holds in the case $\mu(\Psi) = 0$.

Now suppose that $\mu(\Psi) > 0$.  As above, this means that $\mu(S) > 0$.  Consider the set 
$$\hat{\Psi} := \left\{\frac{A}{\mu(\Psi)} \mid A \in \Psi \right\}.$$
Then $\mu(\hat{\Psi}) = 1$ and if we set $\hat{S} = \bigoplus_{A \in \hat{\Psi}} A$, we have $\mu(\hat{S}) = 1$.  For any $m \geq 1$ and any $B \in \hat{\Psi}^m_{\otimes}$, we have
\begin{equation}
\mu(B) \leq \mu(\hat{S}_{\otimes}^m) = \mu(\hat{S})^m = 1.
\end{equation}
It follows from Lemma \ref{lem:bounded} that the semigroup $\mathcal{S}(\hat{\Psi})$ is bounded.  Lemma \ref{lem:monnorm} now implies that there is some monotone norm $\|\cdot\|$ such that 
$$\eta_{\|\cdot\|}(B) \leq \mu(\hat{\Psi}) = 1$$
for all $B \in \hat{\Psi}$.  It now follows from the definition of $\hat{\Psi}$ that
$$\eta_{\|\cdot\|}(A) \leq \mu(\Psi)$$
for all $A \in \Psi$ and hence that 
$$\inf_{\|\cdot\|} \sup_{A \in \Psi} \eta_{\|\cdot \|}(A) \leq \mu(\Psi).$$
Combining this with \eqref{eq:nchar1} completes the proof. $\square$

\emph{Extremal and Barabanov norms}

In conventional algebra, extremal and Barabanov norms play an important role in the analysis of the joint spectral radius and of the stability of difference inclusions \cite{Wirth1, Wirth2, Jun1,Koz1,Gug1}.  We next consider max-algebraic versions of these concepts. 

The norm $\nu$ is said to be \emph{extremal} for a set $\Psi$ of matrices in $\mathbb{R}^{n\times n}_+$ if 
\begin{equation}
\label{eq:Ext1} \nu(A \otimes x) \leq \mu(\Psi) \nu(x) \quad \forall A \in \Psi, x \in \mathbb{R}^n_+.
\end{equation}
If in addition, for every $x \in \mathbb{R}^n_+$, there is some $A \in \Psi$ with 
\begin{equation}
\label{eq:Bar1} \nu(A \otimes x) = \mu(\Psi) \nu(x),
\end{equation}
the norm is said to be a \emph{Barabanov norm} for $\Psi$.

A careful examination of the proof of Proposition \ref{prop:Norm} reveals that it establishes the existence of an extremal norm for any bounded set $\Psi \subseteq \mathbb{R}^{n \times n}_+$ with $\mu(\Psi) > 0$.  Our next result develops on this by providing a max-algebraic version of the recent work in \cite{MasWir1}; this latter paper showed that, for the conventional algebra, extremal norms always exist for compact sets of nonnegative matrices generating an irreducible semigroup.  For the max algebra, under analogous hypotheses, a Barabanov norm exists.  Moreover, this norm can be explicitly characterised in max-algebraic spectral terms. 

\begin{theorem}
\label{thm:Ext} Let $\Psi$ be a compact set in $\mathbb{R}^{n \times n}_+$ such that the semigroup $\mathcal{S}(\Psi)$ is irreducible.  There exists a monotone \emph{Barabanov} norm $\nu$ for $\Psi$ of the form
\begin{equation}\label{eq:barform}
\nu(x) = v^T \otimes |x|
\end{equation}
where $v \gg 0$.  
\end{theorem}
\textbf{Proof:}  As $\Psi$ generates an irreducible semigroup, it follows from Lemma \ref{lem:irred} that there is some irreducible $A$ in $\textrm{conv}_{\otimes}(\Psi)$.  This in turn implies that the matrix 
$$S = \bigoplus_{A \in \Psi} A$$
and hence $S^T$ is irreducible.  

From Proposition \ref{prop:PF1}, there exists some vector $v \gg 0$ such that 
\begin{equation}
\label{eq:PF1} v^T \otimes S = \mu(S) v^T = \mu(\Psi) v^T.
\end{equation}
We claim that defining $\nu(x) = v^T \otimes x$ for $x \in \mathbb{R}^n_+$ and $\nu(x) = \nu(|x|)$ in general, yields a monotone Barabanov norm for $\Psi$.  The properties of a monotone norm follow readily from the definition, noting that $v \gg 0$. 

First note that for all $A \in \Psi$, $x \in \mathbb{R}^n_+$:
$$A \otimes x \leq S \otimes x$$
from which it follows immediately that
\begin{eqnarray*}
\nu(A \otimes x ) &=& v^T \otimes A \otimes x \\
&\leq& v^T \otimes S \otimes x \\
&=& \mu(\Psi) v^T \otimes x \\
&=& \mu(\Psi) \nu(x).
\end{eqnarray*}
Thus $\nu(\cdot)$ is certainly an extremal norm for $\Psi$.  

To see that it is in fact a Barabanov norm, let $x > 0$ be given.  We need to show that there is some $A \in \Psi$ with $$\nu(A \otimes x) = \mu(\Psi) \nu(x).$$
First note that as 
$$\mu(\Psi) \nu(x) = \mu(S) v^T \otimes x = v^T \otimes S \otimes x$$
it is enough to show that there is some $A$ in $\Psi$ with 
$$v^T \otimes A \otimes x = v^T \otimes S \otimes x.$$
This follows as 
\begin{eqnarray*}
v^T \otimes S \otimes x &=& \max_{i, j} \{v_i s_{ij} x_j\} \\
	&=& v_p s_{pq}x_q
\end{eqnarray*}
for some indices $p$, $q$.  Now as $\Psi$ is compact, there is some $A$ in $\Psi$ with $a_{pq} = s_{pq}$ so that 
$$v_p a_{pq}x_q = v_p s_{pq}x_q = \max_{i, j} \{v_i s_{ij} x_j\}.$$
It follows readily that 
\begin{eqnarray*}
v^T \otimes A \otimes x &=& \max_{i, j} \{v_i a_{ij} x_j\}\\
		&=& v_p a_{pq}x_q\\
        &=& v^T \otimes S \otimes x
\end{eqnarray*}
as required. $\square$

It is also possible to prove a converse to the above result that echoes known results on Barabanov norms for the conventional algebra. 

\begin{proposition}
\label{prop:Barabconv} Let $\Psi$ be a compact set in $\mathbb{R}^{n \times n}_+$ such that the semigroup $\mathcal{S}(\Psi)$ is irreducible.  If there exists a monotone norm $\nu(\cdot)$ such that 
\begin{equation}\label{eq:conv1}
\mu \nu(x) = \max_{A \in \Psi} \nu(A \otimes x), \, \forall x \in \mathbb{R}^n_+
\end{equation}
then $\mu = \mu(\Psi).$
\end{proposition}
\textbf{Proof:} First note that for all $A \in \Psi$ and any $x$, $A \otimes x \leq S \otimes x$ where $S$ is given by \eqref{eq:S}.  Lemma \ref{lem:irred} implies that $S$ is irreducible; if we choose $x$ to be a right eigenvector for the irreducible $S$, then $x \gg 0$ and
\begin{eqnarray*}
\mu \nu(x) &=& \max_{A \in \Psi} \nu(A \otimes x) \\
&\leq& \nu (S \otimes x) \\
&=& \mu(S) \nu(x) = \mu(\Psi) \nu(x),
\end{eqnarray*}
where the first inequality follows as the norm $\nu$ is monotone and $A \otimes x \leq S \otimes x$.  As $\nu(x) > 0$, it now follows that $\mu \leq \mu(\Psi)$.  

On the other hand, it follows from \eqref{eq:conv1} that for all $x > 0$
$$\mu = \max_{A \in \Psi} \frac{\nu(A \otimes x)}{\nu(x)}.$$   This implies that for every $A \in \Psi$ and every $x > 0$
$$\frac{\nu(A \otimes x)}{\nu(x)} \leq \mu$$
which implies that  
$$\max_{A \in \Psi} \eta_{\nu}(A) \leq \mu.$$
It now follows immediately that 
$$\mu(\Psi) = \inf_{\|\cdot\|}\max_{A \in \Psi} \eta_{\|\cdot\|}(A) \leq \mu.$$
Thus $\mu = \mu(\Psi)$ as claimed.  $\square$

\emph{Comment on the reducible case}

It is natural to ask what happens in the case where the semigroup $\mathcal{S}(\Psi)$ is not irreducible.  We now describe the circumstances in which a Barabanov norm can fail to exist in this case.  

We first note that if the semigroup is reducible it follows that the matrix $S$ given by \eqref{eq:S} is reducible.  Thus, there exists some permutation matrix $P$ such that the Frobenius normal form of $S$, $PSP^T$ takes the form
\begin{equation}
\label{eq:SPerm}
PSP^T = \left(\begin{array}{c c c c}
			A_{11} & 0 & \ldots & 0 \\
            A_{21} & A_{22} & \ldots & 0 \\
            \vdots & \vdots & \ddots & 0 \\
            A_{p1} & A_{p2} & \ldots & A_{pp}
\end{array}\right).
\end{equation}
It is simple to verify that $P \otimes S \otimes P^T = PSP^T$; it follows from Proposition \ref{prop:simil} that $\mu(\Psi) = \mu(PSP^T)$.  Furthermore, Theorem 3 of \cite{Bapat} implies that 
\begin{itemize}
\item[(i)] $\mu(S) = \max_{1\leq i \leq p} \mu(A_{ii})$;
\item[(ii)] $\mu(A_{ii})$ is an eigenvalue of $S$ if and only if class $j$ does not communicate with class $i$ for any $j$ with $\mu(A_{jj}) > \mu(A_{ii})$.
\end{itemize}
Suppose that there is a class $i$ such that class $j$ does not communicate with class $i$ for any $j$ with with $\mu(A_{jj}) \neq \mu(A_{ii})$, and $\mu(A_{ii}) < \mu(S)$.  It follows that there is some $x > 0$ such that $S \otimes x = \mu(A_{ii}) x$.  Therefore for any monotone norm $\|\cdot \|$, any $A$ in $\Psi$ and this choice of $x$:
$$\| A \otimes x\| \leq \|S \otimes x\| = \mu(A_{ii}) \|x\| < \mu(\Psi) \|x\|.$$
This shows that in this case, a Barabanov norm cannot exist for $\Psi$. 

\subsection{Finiteness Property}
\label{subsec:finprop}
It was conjectured by Lagarias and Wang \cite{LagaWang95} for conventional matrix algebra, that the joint spectral radius of a finite set of matrices is attained by a finite product of matrices from the set.  This is now known not to hold in general. When it does hold we say that the set of matrices has the {\em finiteness property}.

As an easy consequence of the existence of a Barabanov norm of the type
described in \eqref{eq:barform}, we obtain a finiteness property for the max algebraic joint spectral radius even for compact sets of matrices. In the max algebraic case it even holds that the length of the product attaining the joint spectral radius may be bounded by the dimension of the state space.

It should be noted here that this result for the case of irreducible finite sets of matrices can essentially be found in \cite[Theorem~2]{Gaub}.  However the terminology and notation used in this earlier paper is that of discrete event systems and the result may not be immediately apparent to a linear algebraic audience.  We present a linear algebraic statement of the result for the case of \emph{compact} sets of matrices that are not necessarily irreducible. The proof is closely related to ideas presented in \cite{Gurv95}, where the finiteness property is shown in the conventional matrix algebra setting for finite sets of matrices which admit a polytopic extremal norm.

\begin{theorem}
Let $\Psi$ be a compact set in $\R_{+}^{n\times n}$. Then there exist $1\leq k \leq n$ and $A_1, \ldots, A_k \in \Psi$ such that
 \begin{equation*}
 \mu\left(\Psi\right)=\mu\left(A_k\otimes \cdots \otimes A_1\right)\,.
 \end{equation*}
\end{theorem}

\begin{proof}
We first show the claim under the additional assumption that the semigroup
$\mathcal{S}(\Psi)$ is irreducible. In this case, without loss of
generality, we may assume $\mu(\Psi)=1$.

Let $v=\left(v_1, \ldots, v_n\right)\gg 0$ be such that
$\nu(x):=v^{\top}\otimes \vert x\vert$ defines a Barabanov norm for
$\Psi$. The part of the boundary of the unit ball of $\nu$ contained in
$\R_{+}^{n}$ is given as the union of the sets
\begin{equation*}
R_i:=\{x\in \R_{+}^{n}\, \lvert \, v_ix_i=1,\,v_jx_j\leq 1, \,j\neq
i\}\,,\quad i=1, \ldots, n\,.
\end{equation*}
Let $x\in R_i\setminus \bigcup_{j\neq i}R_j$ then $v_ix_i=1>v_jx_j$ for $j\neq i$. 
As $\nu$ is a Barabanov norm, there exists an $A\in \Psi$ such that $A\otimes x\in R_{\ell}$ for some $\ell\in \{1, \ldots, n\}$.
Now $\nu$ is extremal for $\Psi$, so that $v_ia_{ij}x_j\leq 1$ if $v_jx_j\leq 1$ as otherwise $\mu( A\otimes z)>1$ for some $z\in R_j$. This implies
\begin{equation*}
a_{ij}\leq \frac{v_j}{v_i}\quad \forall \, 1 \leq i,j\leq n\,.
\end{equation*}
On the other hand $A\otimes x \in R_{\ell}$ implies
\begin{equation*}
1=\max_{j=1,\ldots,n} \{ v_{\ell}a_{\ell j}x_j \} \leq \max_{j=1,\ldots,n} \{ v_jx_j \} =v_ix_i=1\,.
\end{equation*}
As equality throughout is obtained only for $j=i$, we obtain $a_{\ell i}=\tfrac{v_i}{v_{\ell}}$ and so $A\otimes R_i\subset R_{\ell}$. 
We thus see that for any $i\in \{1, \ldots, n\}$ there exists an $A_i\in
\Psi$ such that $A_i\otimes R_i\subset R_{\ell(i)}$ for some $\ell(i)\in
\{1, \ldots, n\}$. 

If we consider the graph $G(V,E)$ with nodes $V=\left(R_1, \ldots,
  R_n\right)$ and edges $(i,j)$ if there exists an $A\in \Psi$ such that
$A\otimes R_i\subset R_j$, the previous argument shows that there is a
path of infinite length in this graph. This path necessarily contains a
cycle of  length $k\leq n$. Thus there exists an $i\in \{1, \ldots, n\}$ and matrices $A_1, \ldots, A_k \in \Psi$ such that 
\begin{equation*}
A_k\otimes \ldots \otimes A_1 \otimes R_i\subset R_i\,.
\end{equation*}
It follows that $\mu\left(A_k\otimes \ldots \otimes A_1\right)=1$, as
desired.

For the case of reducible $\Psi$, note that if $\mu(\Psi)=0$ there is
nothing to show. Otherwise, we can bring all matrices in $\Psi$ into the lower
block triangular form \eqref{eq:SPerm} using a permutation matrix as a similarity
transformation. By Proposition~\ref{prop:simil} this does not change the joint
spectral radius. The previous argument may then be applied to the
diagonal blocks that realize the joint spectral radius.
\end{proof}

\begin{example}
Consider the matrices in $\mathbb{R}^{3 \times 3}_+$ given by
\[
A_1 = \left(\begin{array}{c c c}
			1/3 & 1/2 & 1\\
            3/4 & 2/3 & 1/5 \\
            3/5 & 1/5 & 0\end{array}\right), \;\;
A_2 = \left(\begin{array}{c c c}
			0 & 1/4 & 1/2\\
            0 & 4/5 & 10/3\\
            1/4 & 0 & 1/4
            \end{array}\right).
\]
Then by directly computing $S = A_1 \oplus A_2$, it is not difficult to see that $\mu(\{A_1, A_2\}) = 1$.  The product that realises this is $A_1\otimes A_2 \otimes A_1$ which is given by:
\[A_1\otimes A_2 \otimes A_1 = \left(\begin{array}{c c c}
						1 & 1/3 & 1/4\\
                        4/3 & 20/45 & 8/75\\
                        1/4 & 2/15 & 4/125\end{array}\right),
\]
which has $\mu(A_1A_2A_1) = 1$.
\end{example}

\subsection{Lipschitz Continuity}
It was established in \cite{MasWir1} that, for the conventional algebra, the joint spectral radius is locally Lipschitz continuous on the space of compact irreducible subsets of nonnegative matrices endowed with the Hausdorff metric.  While there have been several papers extending the continuity of the JSR to the max algebraic setting \cite{Lur2, Pep2, Pep3}, none of these address the question of whether it is in fact Lipschitz continuous.  We shall close this gap in the current subsection.  

We first recall the definition of the Hausdorff distance.  Let $\beta$ denote the collection of all compact subsets $\Psi$ of $\mathbb{R}^{n \times n}_+$ such that the semigroup $\mathcal{S}(\Psi)$ is irreducible.  

Given a norm $\|\cdot\|$ on $\mathbb{R}^{n \times n}$ we define the Hausdorff distance $\mathcal{H}_{\|\cdot \|}(\Psi, \Phi)$ between $\Psi$ and $\Phi$ in $\beta$ to be
\begin{equation}
\label{eq:Haus} H_{\|\cdot \|}(\Psi, \Phi) := \max \left\{\max_{A \in \Psi}\{\textrm{dist}(A, \Phi)\}, \max_{B \in \Phi}\{\textrm{dist}(B, \Psi)\}\right\},
\end{equation}
where $\textrm{dist}(A, \Phi) := \min_{B \in \Phi} \|A - B\|$.  

As in \cite{MasWir1, Wirth1}, the concept of \emph{eccentricity} of norms will play a central role in our argument here.  The eccentricity $\textrm{ecc}_{\|\cdot \|}(\nu)$ of a norm $\nu$ (on $\mathbb{R}^n$) with respect to a norm $\|\cdot\|$ is defined by 
\begin{equation}\label{eq:ecc1}
\textrm{ecc}_{\|\cdot \|} (\nu) : = \frac{\max\{ \nu(x) \mid \| x\| = 1\}}{\min\{ \nu(x) \mid \| x\| = 1\}}.
\end{equation}

In what follows, we shall study the eccentricity of Barabanov norms with respect to the standard $\|\cdot\|_{\infty}$ norm on $\mathbb{R}^n$ and shall simply write $\mathcal{H}$ for the Hausdorff metric on $\beta$ generated by the matrix norm induced by $\|\cdot\|_{\infty}$.
In analogy with the arguments presented in \cite{MasWir1}, we shall show that the eccentricity of Barabanov norms is bounded on compact subsets of the metric space $(\beta, \mathcal{H})$.  Our proof is made a little easier by exploiting the specific form of Barabanov norm whose existence is ensured by Theorem \ref{thm:Ext}.  This will then be used to show that the joint spectral radius $\mu$ is locally Lipschitz continuous on $(\beta, \mathcal{H})$.  It is important to note that we need to prove the boundedness property for Barabanov norms with respect to the max algebra, and hence that the result of \cite{MasWir1} cannot be directly applied here.  

\begin{proposition}
\label{prop:ecc}
Let $\mathcal{X} \subseteq \beta$ be a compact subset of $(\beta, \mathcal{H})$.  Then there exists a constant $C$ such that for every $\Psi \in \mathcal{X}$, there is a Barabanov norm $\nu(\cdot)$ for $\Psi$ with $\textrm{ecc}_{\|\cdot\|}(\nu) \leq C$.
\end{proposition}
\textbf{Proof:}  As in the proof of Proposition 4.2 of \cite{MasWir1}, we show this property for a neighbourhood of $\Psi \in\mathcal{X}$.  The result then follows readily using the compactness of $\mathcal{X}$.  

So let $\Psi \in \mathcal{X}$ be given.  We claim that there is some neighbourhood $N$ of $\Psi$ and some constant $C$ such that for all $\Phi \in N$, there is a Barabanov norm $\nu$ for $\Phi$ with $\textrm{ecc}_{\|\cdot \|}(\nu) \leq C$.  If this is not the case, we can choose a sequence of sets $\Psi_k$ in $\beta$ with $\Psi_k \rightarrow \Psi$ and a sequence of constants $C_k \rightarrow \infty$ such that every Barabanov norm for $\Psi_k$ has eccentricity greater than $C_k$.  

If we write $S_k$ for the matrix $S(\Psi_k)$, then by Theorem~\ref{thm:Ext} any left max-eigenvector of $S_k$ defines a Barabanov norm for $\Psi_k$.  Now choose an eigenvector $v^{(k)}$ of $S_k$ with $\|v^{(k)}\|_{\infty} = 1$; then 
the eccentricity of the associated Barabanov norm $\nu_k$ is given by
$$\textrm{ecc}_{\|\cdot \|}(\nu_k) = \frac{1}{\min_i v^{(k)}_i}.$$
As $\|v^{(k)}\|_{\infty} = 1$ for all $k$, by passing to a subsequence if necessary, we can assume that $v^{(k)} \rightarrow v$ for some $v > 0$ with $\|v\|_{\infty} = 1$.  It is also not hard to see that for this $v$, $\min_i v_i = 0$. 

Next note that as $\Psi_k \rightarrow \Psi$ in the Hausdorff metric, $S_k \rightarrow S$.  Furthermore, by the continuity of the joint spectral radius \cite{Lur2, Pep2, Pep3}, $\mu(\Psi_k) \rightarrow \mu(\Psi)$.  Thus for every $k$ we have that 
\begin{equation}
\label{eq:Lip1} (v^{(k)})^T \otimes S_k = \mu(S_k) (v^{(k)})^T.
\end{equation}
Moreover, as $v^{(k)} \rightarrow v$, $S_k \rightarrow S$ and $\mu(S_k) \rightarrow \mu(S)$, it follows that
\begin{equation}
\label{eq:Lip2} v^T \otimes S = \mu(S) v.
\end{equation}
However, as $v_i = 0$ for some $i$ and $S$ is irreducible, this contradicts Proposition \ref{prop:PF1}; this shows that the claim is true and the result is proven. $\square$

We can now use the previous result to establish the following.

\begin{theorem}
\label{thm:Lip} The joint spectral radius $\mu: \beta \rightarrow \mathbb{R}_+$ is locally Lipschitz continuous on $\beta$ with respect to the Hausdorff metric.
\end{theorem}
\textbf{Proof:} Let $\mathcal{X}$ be a compact subset of $\beta$ and let $\Psi$, $\Phi$ be any two elements of $\mathcal{X}$.  We can choose two absolute Barabanov norms $\nu_{\Psi}(\cdot)$, $\nu_{\Phi}(\cdot)$ such that the eccentricity of both is less than some constant $C$.  It follows that 
\begin{eqnarray}
\label{eq:eccBar} 
\mathcal{H}_{\nu_\Psi}(\Psi, \Phi) &\leq& C \mathcal{H}(\Psi, \Phi)\\
\nonumber \mathcal{H}_{\nu_\Phi}(\Psi, \Phi) &\leq& C \mathcal{H}(\Psi, \Phi).
\end{eqnarray}
Given any $B \in \Phi$, we can select $A \in \Psi$ with $\eta_{\nu_\Psi}(A - B) = \textrm{dist}_{\nu_\Psi}(B, \Psi)$.  It now follows that
$$\eta_{\nu_\Psi}(B) \leq \eta_{\nu_{\Psi}}(A) + \textrm{dist}_{\nu_\Psi}(B, \Psi) \leq \mu(\Psi) + \mathcal{H}_{\nu_\Psi}(\Psi, \Phi).$$
It now follows from \eqref{eq:eccBar} that
\begin{equation}
\label{eq:eccBar1} \eta_{\nu_\Psi}(B) \leq \mu(\Psi) + C \mathcal{H}(\Psi, \Phi).
\end{equation}
As this is true for all $B \in \Phi$, Proposition \ref{prop:Norm} implies that 
\begin{equation}
\mu(\Phi) \leq \mu(\Psi) + C \mathcal{H}(\Psi, \Phi).
\end{equation}
Repeating the same argument with the Barabanov norm $\nu_{\Phi}$ we see that
$$|\mu(\Phi) - \mu(\Psi)| \leq C \mathcal{H}(\Psi, \Phi)$$
as claimed. $\square$

\section{Regularity of the max-algebraic spectral radius and joint spectral radius}
\label{sec:hoelder}

As an alternative approach to the analysis of various continuity properties of the max-algebraic joint spectral radius we now present results which are based on properties of the max-algebraic spectral radius. To this end, we will first derive some perturbation theoretic results of the max-algebraic spectral radius which are to the best of our knowledge new.\\
Let $(X,d)$ be a metric space. Recall that a map $f:X\mapsto \R$ is called Hoelder continuous (of order $\alpha >0$) if there exists a constant $C \geq 0$ such that
\begin{equation}
\label{Hoe}
|f(x)-f(y)|\leq Cd(x,y)^{\alpha}
\end{equation}
for all $x,y \in X$. The map $f$ is called locally Hoelder continuous (of order $\alpha$) if for every $z \in X$ there exists an $\epsilon >0$ and a $C\geq 0$ such that \eqref{Hoe} holds for all $x,y \in B_{\epsilon}(z)$. 
We first point out continuity properties of the max-spectral radius. 
To this end, we need the following lemma.

\begin{lemma}
\label{lem:sqrt}
Let $n\in \N$. The map $f:\R_{+}^{n}\rightarrow \R$
\begin{equation*}
\left(x_1, \ldots, x_n\right) \mapsto \left(x_{1}\cdot \ldots \cdot x_n\right)^{\frac{1}{n}}
\end{equation*}
is differentiable on $\{x\in \R_{+}^{n} | x\gg 0\}$ and locally Hoelder continuous of order $\frac{1}{n}$ on $\R_{+}^{n}$.
\end{lemma}

\begin{proof}
The claim for differentiability on the interior of $\R_{+}^{n}$ is
well-known. The map $f_1:\left(x_1, \ldots, x_n\right) \mapsto
\prod_{i=1}^{n}x_i$ is differentiable and thus locally Lipschitz, while
the map $f_2:z\mapsto \sqrt[n]{z}$ is clearly locally Hoelder continuous of order
$1/n$ on $[0, \infty)$. Let $L$ be a local Lipschitz constant for $f_1$ in
a bounded subset $U$ of $\R_{+}^{n}$ and let $H$ be a Hoelder constant for
$f_2$. Then we have for $x,y\in U$ that
\begin{equation*}
|f(y)-f(x)|=| f_2\circ f_1(y)-f_2\circ f_1(x)| \leq H\left|f_1(y)-f_1(x)\right|^{1/n}\leq HL^{\frac{1}{n}}\|y-x\|^{\frac{1}{n}}\,.
\end{equation*}
This shows the claim.
\end{proof}

With this result we obtain
\begin{proposition}
\label{prop:msr-reg}
The map $\mu: \R_{+}^{n\times n}\rightarrow \R_{+}$ has the following properties
\begin{enumerate}[(i)]

\item If there is a unique cycle $(a_{i_1i_2}, \ldots,a_{i_ki_1})$ in $A$ such that
\begin{equation*}
\mu (A)= \left(a_{i_1i_2}\cdot \ldots \cdot a_{i_ki_1}\right)^{\frac{1}{k}}
\end{equation*}
then $\mu$ is differentiable at $A$.

\item $\mu$ is locally Lipschitz continuous on the set 
\begin{equation*}
X:=\{A\in \R_{+}^{n\times n}\mid \mu (A)>0\}\,.
\end{equation*}

\item $\mu$ is locally Hoelder continuous of order $\frac{1}{n}$ on $\R_{+}^{n\times n}$.
\end{enumerate}
\end{proposition}

\begin{proof}
(i) If the unique cycle exists as in the asssumption then $\mu (A)>0$, as otherwise all cycle products evaluate to zero. By continuity of the cycle products there exists an open neighborhood $U$ of $A$ such that
\begin{equation*}
\mu(B)=\left(b_{i_1i_2}\cdot \ldots \cdot b_{i_ki_1} \right)^{\frac{1}{k}}, \quad \forall B\in U\,.
\end{equation*}
The claim now follows by applying Lemma \ref{lem:sqrt} to this particular
cycle mean.

(ii) If $\mu (A)>0$, there exist finitely many cycles $\tau_1, \ldots, \tau_m$, where $\tau_j=\left(i_{j1}, \ldots, i_{jk(j)} \right)$, such that
\begin{equation*}
\mu(A)=\max_{j=1, \ldots, m} \left(a_{i_{j1}i_{j2}}\cdot \ldots \cdot a_{i_{jk(j)}i_{j1}}\right)^{\frac{1}{k(j)}}\,.
\end{equation*}
By continuity there exists a neighborhood $U$ of $A$ such that for all $B \in U$ we have
\begin{equation*}
\mu(B)=\max_{j=1, \ldots, m} \left(b_{i_{j1}i_{j2}}\cdot \ldots \cdot b_{i_{jk(j)}i_{j1}}\right)^{\frac{1}{k(j)}}\,.
\end{equation*}
As the maximum of differentiable functions is locally Lipschitz continuous, the claim follows.

(iii) The claim follows from Lemma~\ref{lem:sqrt} as the maximum of finitely many locally Hoelder continuous functions is again Hoelder continuous.
\end{proof}

We now apply Proposition~\ref{prop:msr-reg}  to obtain alternative proofs for
regularity properties of the max-algebraic joint spectral radius and in
fact a stronger statement. 
For $\Psi \subset \R^{n \times n}_+$ recall from
\eqref{eq:Sres} that
\begin{equation*}
\mu(\Psi)=\mu\left(S(\Psi)\right).
\end{equation*}
We continue to use the notation introduced prior Theorem~\ref{thm:Lip}.

\begin{theorem}
\label{theo:jsr-Hoelder}
\begin{enumerate}[(i)]
\item The joint spectral radius is locally Lipschitz continuous on the set
\begin{equation*}
\overline{\beta}:=\{\Psi \subset \R_{+}^{n\times n} \mid \Psi \text{ is compact}, \mu(\Psi)>0\}\,.
\end{equation*}

\item The joint spectral radius is Hoelder continuous on the space of compact subsets of $R_{+}^{n\times n}$ endowed with the Hausdorff metric.
\end{enumerate}
\end{theorem}

\begin{proof}
We first note that the map
\begin{equation}
\label{eq:Smap}
\Psi \mapsto S(\Psi)=\bigoplus_{A\in \Psi}A
\end{equation}
is locally Lipschitz continuous from the space of compact subsets of
$\R_{+}^{n\times n}$ (endowed with the Hausdorff metric) to $\R_{+}^{n\times n}$ with Lipschitz constant $1$.
The claims (i) and (ii) follow directly from the previous observation together with Proposition \ref{prop:msr-reg}, as the composition of locally Lipschitz functions is locally Lipschitz and the composition of a Hoelder continuous function with a locally Lipschitz continuous function is again Hoelder continuous.
\end{proof}

\begin{remark}
As the Lipschitz constant of the map in \eqref{eq:Smap} is $1$, the Lipschitz/Hoelder constants of the max-algebraic joint spectral radius at $\Psi$ are essentially given by those of the max-algebraic radius at $S(\Psi)$.
\end{remark}

\section{Monotonicity Properties}
\label{sec:mon}
As noted in Section \ref{sec:bg}, the analogue of a linear space in the max algebraic setting is a \emph{max cone} as defined by \eqref{eq:maxspandef}.  It follows immediately from the definition given in \eqref{eq:maxspandef} that if $A \in \textrm{span}_{\otimes}(\Psi)$, where $\Psi \subseteq \mathbb{R}^{n \times n}_+$ then $\lambda A \in \textrm{span}_{\otimes}(\Psi)$ for any $\lambda > 0$.  

We will work the usual topology on $\mathbb{R}^n_+$ which is the relative topology inherited from the Euclidean topology on $\mathbb{R}^n$.  For convenience we shall use the norm $\|x\|_\infty = \max_i |x_i|$ as the norm generating this topology as it is particularly well suited to the max algebraic operations.  

For a max-convex subset $C$ of $\mathbb{R}^{n \times n}_+$, we consider the interior of $C$ with respect to the relative topology on $\textrm{span}_{\otimes}(C)$ and denote this set by $\textrm{int}_{\otimes}(C)$.

\begin{proposition}
\label{prop:increasing} Let $\Psi_1$, $\Psi_2$ be compact sets of nonnegative matrices and assume that $\mathcal{S}(\Psi_2)$ is irreducible.  Then 
\begin{equation}
\label{eq:Inc1} \textrm{conv}_{\otimes}(\Psi_1) \subseteq \textrm{int}_{\otimes}(\textrm{conv}_{\otimes}(\Psi_2)) \Rightarrow \mu(\Psi_1) < \mu(\Psi_2).
\end{equation}
\end{proposition}
\textbf{Proof:}  As $\mathcal{S}(\Psi_2)$ is irreducible, the matrix $S_2 = S(\Psi_2)$ is irreducible and hence $\mu(\Psi_2) > 0$.  Thus, the result is immediate if $\mu(\Psi_1) = 0$.  So we now assume that $\mu(\Psi_1) > 0$.

As $\Psi_1$ is compact, it follows that the matrix $S_1 = \bigoplus_{A \in \Psi_1} A$ lies in $\textrm{conv}(\Psi_1)$; thus $S_1 \in \textrm{int}_{\otimes}(\textrm{conv}(\Psi_2)) $.  We claim that there is some $\lambda > 1$ such that $\lambda S_1 \in \textrm{conv}(\Psi_2).$  To see this, note that $$\lambda S_1 \in \textrm{span}(\textrm{conv}(\Psi_2))$$ for all $\lambda > 0$.  Moreover, there is some $\delta > 0$ such that 
$$B(S_1, \delta) \cap \textrm{span}(\textrm{conv}(\Psi_2)) \subseteq \textrm{conv}(\Psi_2).$$
Clearly we can choose $\delta_1 > 0$ sufficiently small to ensure that with $\lambda = 1 + \delta_1$, $\lambda S_1 \in B(S_1, \delta)$.  As $\lambda S_1 \in \textrm{span}(\textrm{conv}(\Psi_2))$, it follows that $\lambda S_1 \in \textrm{conv}(\Psi_2)$ as claimed.  

It now follows readily that $$\lambda S_1 \leq S_2 = \bigoplus_{A \in \Psi_2} A$$ and hence that 
$$\mu(\Psi_1) = \mu(S_1) < \lambda \mu(S_1) \leq \mu(S_2) = \mu(\Psi_2)$$
as required. 

\textbf{Remark:}

In \cite{Wirth2}, it was shown that the joint spectral radius was a strictly increasing function in the sense that 
$$\textrm{conv}(\Psi_1) \subseteq \textrm{ri}(\textrm{conv}(\Psi_2))$$
and irreducibility of $\Psi_2$ imply that $\rho(\Psi_1) < \rho(\Psi_2)$. Here $\textrm{conv}$ denotes the conventional convex hull, $\textrm{ri}$ denotes the relative interior of a convex set (interior relative to its affine hull), and $\rho$ is the conventional joint spectral radius.  Our previous result is reminiscent of this fact.  However, our concept of $\textrm{int}_{\otimes}(C)$ is not an analogue of the relative interior.  In fact, it is possible for $\textrm{int}_{\otimes}(C)$ to be empty in contrast to the relative interior in classical convex analysis. 

\section{Concluding Remarks}
\label{sec:conc}
We have shown that a simple cone-theoretic irreducibility condition is sufficient for the existence of Barabanov norms in the max algebra and, moreover, have given an explicit description of the norm in this case.  We then used this result and the form of the norm to prove Lipschitz continuity of the max algebraic JSR with respect to the Hausdorff metric.  We have also extended results on the monotonicity of the JSR to the max algebraic setting.  An interesting direction for future work would be to investigate infinite dimensional extensions of these results to idempotent semi-modules. 

%\section{Finiteness Conjecture in Max Algebra}
\section*{Acknowledgments}
The first author thanks the INdAM GNCS for financial support and also
Gran Sasso Science Institute (GSSI, L'Aquila).  The work of the second named author was partially supported by Science Foundation Ireland grant 13/RC/2094 and co-funded under the European Regional Development Fund through the Southern \& Eastern Regional Operational Programme  to Lero - the Irish Software Research Centre (www.lero.ie).
%\section{Inhomogeneous Products and Convergence}

\end{document}